\documentclass[twoside]{article}
\usepackage{amsfonts}
\usepackage{amstext}
\usepackage{amsthm}
\usepackage{amsmath}
\usepackage{eufrak}
\usepackage{mathrsfs}
\usepackage{amssymb}
\usepackage{latexsym}
\usepackage[latin1]{inputenc}

\newtheorem{teo}{Theorem}[section]

\newtheorem{prop}{Proposition}[section]
\newtheorem{defi}{Definition}[section]

\newtheorem{exe}{Exemple}[section]
\newtheorem{obs}{Remark}[section]
\newcommand{\p}{\mathbb{P}_{\mathbb{C}}^n}
\newcommand{\pd}{\mathbb{P}_{\mathbb{C}}^2}

\newcommand{\D}{\mathbb{P}(\varpi_0,\varpi_1,\varpi_2)}
\newcommand{\F}{{\mathcal{F}}}
\begin{document}

\newtheorem{theorem}{Theorem}
\newtheorem{proposition}[theorem]{Proposition}
\newtheorem{lemma}[theorem]{Lemma}

\newtheorem{definition}{Definition}

\title{On algebraic hypersurfaces invariant by  weighted
projective foliations }
\author{
Maurício Corrêa JR\\
\small{Departamento de Matem\'atica}\\
\small{Universidade Federal de Minas Gerais}\\
\small{30123-970 Belo Horizonte - MG, Brasil}\\
\small{\texttt{mauriciojr@ufmg.br}}}
\date{\today}
\maketitle
\begin{abstract}
In this work we study some problems related with algebraic
hypersurfaces invariant by foliations on weighted projective spaces
$\mathbb{P}_{\mathbb{C}}(\varpi_0,\dots,\varpi_n)$ generalizing some
results known for $\p$, as for example: the number of singularities,
with multiplicities, contained in the invariant quasi-smooth
hypersurfaces ; Poincaré problem on weighted projective plane. Also,
we show that there exist numbers $M(d), N(k)\in \mathbb{N}$, such
that if a foliation $\F$ of degree $d$ has a hipersurface invariant
of degree $k>M(d)$, then is either $\F$ possesses a rational
integral first or the number of $\F$-invariants hypersurfaces of
degree $k$ is at most $N(k).$
\end{abstract}

\section{Introduction}

In  the end of the nineteenth century Darboux \cite{Da}, Poincaré
\cite{P}, Painlevé \cite{Pa} and Autonne \cite{Au} had given
beginning to the study of the problem of deciding whether a
holomorphic foliation $\F$ on $\pd$ is algebraically integrable,
i.e, if $\F$ admit a rational first integral. In \cite{Da} Darboux
showed that if a foliation on $\mathbb{C}^2$ has enough algebraic
solutions then it must have a first integral. Some years later
Jouanolou in \cite{J} improved this theory to obtain rational first
integrals for foliations on $\pd$. More precisely, He proved that if
$\F$  admit
$$\frac{d(d+1)}{2}+2$$ invariants algebraic curves then $\F$ admit a
rational first integral. Using the  same  arguments it is possible
to show that, in general for the weighted projective plane, we have:
if $\F$ is a foliation on $\mathbb{P}(\varpi_0,\varpi_1,\varpi_2)$
of degree $d$ which admits
$$dim_{\mathbb{C}}H^0(\mathbb{P}(\varpi_0,\varpi_1,\varpi_2),\mathscr{O}(d))+2$$
invariants algebraic curves, then $\F$ admit a rational first
integral.

The Darboux-Jouanolou theory of integrability provides a link
between the algebraic integrability foliations on $\p$ and the
number of invariant algebraic hypersurfaces that they have. In this
direction,  J. V. Pereira in \cite{Pe} approached this subject using
the concept of  \emph{extatic curve} of a foliation on $\pd$ with
respect to the a linear system. He got the following result:
\\
\\
\textbf{Theorem.}\cite{Pe}\  \emph{Let $\F$ be a foliation on
$\mathbb{P}_{\mathbb{C}}^2$, of degree $d\geq2$, that does not admit
rational first integral of degree $\leq k$. Then $\F$  has at most}
$$
{k+2\choose k}+\frac{(d-1)}{k} \cdot{{k+2\choose k}\choose 2}.
$$
\emph{invariant curves of degree $k$.}
\\

Using the extatic hypersurface we get the following result.

\begin{teo}\label{2}
Let $\F$ be a one-dimensional foliation on
$\mathbb{P}(\varpi_0,\dots,\varpi_n)$, of degree $d\geq 2$, and
$\mathscr{N}(k)$ the number of hypersurfaces of degree $k$
invariants by $\F$ . If $k>d-1$,  then is either $\F$ possesses a
rational integral first or
$$
\mathscr{N}(k)<
h^0(\mathbb{P}(\varpi),\mathscr{O}_{\mathbb{P}(\varpi)}(k))+
{h^0(\mathbb{P}(\varpi),\mathscr{O}_{\mathbb{P}(\varpi)}(k))\choose
2}
$$

\end{teo}

This  show that there exist numbers $M(d), N(k)\in \mathbb{N}$, such
that if a foliation $\F$ of degree $d$ has a hipersurface invariant
of degree $k>M(d)$, then is either $\F$ possesses a rational
integral first or the number of $\F$-invariants hypersurfaces of
degree $k$ is at most $N(k).$

Henri Poincaré studied in \cite{P} the problem which, in the modern
terminology, says: \emph{"Is it possible to decide if a holomorphic
foliation $\F$ on the complex projective plane
$\mathbb{P}_{\mathbb{C}}^2$ admits a rational first integral ?"}
Poincaré observed that in order to solve this problem is sufficient
to find a bound for the degree of the generic leaf of $\F$. In
general, this is not possible, but doing some hypothesis we obtain
an affirmative answer for this problem, which nowadays is known as
$\emph{Poincaré Problem}$. This problem was treated by Cerveau and
Lins Neto \cite{CN} and by Carnicer \cite{Ca} and the answer is
provided the foliation or the curve are subjected to some
conditions. Simple examples show that, when $S$ is a dicritical
separatrix of $\F$, the search for a positive solution to the
problem is meaningless. The obstruction in this case was given by M.
Brunella which in \cite{B1} observed that obstruction  to the
positive solution to Poincaré problem is given by the GSV index.
Meanwhile  A. Lins Neto constructed, in \cite{N} some remarkable
families of foliations on $\pd$ providing counterexamples for this
problem. We summarize the results obtained in \cite{CN}, \cite{Ca}
and \cite{MS} in the following theorem.

\begin{teo}
Let $\F$ a foliation on  $\pd$ and $S$ a  separatrix. We have that:
\begin{itemize}
  \item [i)]$deg(S)\leq
deg(\F)+2$,  if $S$ is a non-dicritical or

  \item [ii)]if $S$ is smooth them $deg(S)\leq
deg(\F)+1.$

\end{itemize}
\end{teo}

We obtain the following result:

\begin{teo}\label{3}(Poincaré problem)
Let $\F$ be a foliation on  $\D$ such that $Sing(\F)\cap
Sing(\D)=\emptyset$ and $S$ a separatrix. We have that:
\begin{itemize}
  \item [i)] if $S$ is a non-dicritical separatrix, them $deg(S)\leq
deg(\F)+\varpi_0+\varpi_1+\varpi_2 -1.$

  \item [ii)]if $S$ is quasi-smooth them $deg(S)\leq
deg(\F)+\varpi_0+\varpi_1+\varpi_2 -2.$

\end{itemize}
\end{teo}

A hypersurface $\mathscr{V}$ on $\mathbb{P}(\varpi)$ is said
quasi-smooth if the cone $\pi^{-1}(\mathscr{V})$ is smooth on
$\mathbb{C}^{n+1}\backslash\{0\}$. We calculate the number of
singularities of a foliation, with multiplicities, contained in the
invariant quasi-smooth hypersurfaces on
$\mathbb{P}(\varpi_0,\dots,\varpi_n)$, generalizing the M. Soares's
result \cite{MS}.

\begin{teo}\label{number}
Let $\F$ be a foliation of degree $d$ on
$\mathbb{P}(\varpi_0,\dots,\varpi_n)$ with isolated singularities.
If $\mathscr{V}$  is a   quasi-smooth hypersurface on
$\mathbb{P}(\varpi_0,\dots,\varpi_n)$ invariant by a foliation $\F$,
then
$$
(\varpi_0\cdots\varpi_n)\cdot\sum_{p\in Sin(\F)\cap
\mathscr{V}}\mu_p^{orb}(\F)=\displaystyle\sum_{i=0}^{n-1}\left[\sum_{k=0}^{i}(-1)^k\sigma_{i-k}
(\varpi_0,\dots,\varpi_n)deg(\mathscr{V})^{k+1}\right](d-1)^{n-1-i},
$$
where  and $\sigma_j$ is the $j$-th elementary symmetric  function.

\end{teo}

\begin{exe}
Let $\mathbb{P}(1,\dots,1)=\mathbb{P}^n$ be and $\mathscr{V}$ a
algebraic smooth hypersurface. Since  $\sigma_{i-k}
(1,\dots,1)={n+1\choose i-k}$, we have that

$$
\begin{array}{ccl}
  \sum_{p\in Sin(\F)\cap
\mathscr{V}}\mu_p(\F) & = &
\displaystyle\sum_{i=0}^{n-1}\left[\sum_{k=0}^{i}(-1)^k {n+1\choose
i-k}
deg(\mathscr{V})^{k+1}\right](d-1)^{n-1-i} \\
\\
   & = & \displaystyle\sum_{i=0}^{n-1}\left[
1+(-1)^i(deg(\mathscr{V})-1)^{i+1}\right]d^{n-1-i}
\end{array}
$$
This is the M. Soares's result \cite{MS}. If $deg(\mathscr{V})=1$,
it follows that
$$
\sum_{p\in Sin(\F)}\mu_p(\F)=d^{n-1}+\cdots +d+1.
$$
\end{exe}

\section{Weighted projective space $\mathbb{P}(\varpi_0,\dots,\varpi_n)$}

Let $\varpi_0,\dots,\varpi_n$ be integers $\geq 1$ and pairwise
prime. Consider the $\mathbb{C}^*$-action on
$\mathbb{C}^{n+1}\backslash \{0\}$ given by
$$
\lambda\cdot(z_0,\dots,z_n)=(\lambda^{\varpi_0}
z_0,\dots,\lambda^{\varpi_n} z_n),
$$
where $\lambda\in \mathbb{C}^*$ and $(z_0,\dots,z_n)\in
\mathbb{C}^{n+1}\backslash \{0\}$. The quotient space $
\mathbb{P}(\varpi_0,\dots,\varpi_n)=(\mathbb{C}^{n+1}\backslash
\{0\}/\sim) $, induced for the action above, is called
\emph{weighted projective  space} of type
$(\varpi_0,\dots,\varpi_n)$. Some times we will use the notation
$\mathbb{P}(\varpi_0,\dots,\varpi_n):=\mathbb{P}(\varpi)$.

\subsection{Orbifold structures of $\mathbb{P}(\varpi_0,\dots,\varpi_n)$}

Consider the opened $\mathcal{U}_i=\{[z_0:\dots:z_n]\in
\mathbb{P}(\varpi_0,\dots,\varpi_n);\ z_i\neq0 \}\subset
\mathbb{P}(\varpi_0,\dots,\varpi_n)$, with $i=0,1,\dots,n.$ Let
$\mu_{\varpi_i}\subset \mathbb{C}^*$ the subgroup of $\varpi_i$-th
roots of unity. We can define the homeomorphisms
$\phi_i:\mathcal{U}_i\longrightarrow \mathbb{C}^{n}/\mu_{\varpi_i}$,
given by
$$
\phi_i([z_0:\dots:z_n])=\left(\frac{z_0}{z_i^{\varpi_0/\varpi_i,}},\dots,\frac{\widehat{z_i}}{z_i},\dots,\frac{z_n}{z_i^{\varpi_n/\varpi_i,}}
\right)_{\varpi_i}
$$
where the symbol $"\wedge"$ means omission and $(\cdot)_{\varpi_i}$
is a $\varpi_i$-conjugacy class in $\mathbb{C}^{n}/\mu_{\varpi_i}$
with $\mu_{\varpi_i}$ acting on $\mathbb{C}^{n}$ by
$$
\lambda\cdot(z_0,\dots,z_n)=(\lambda^{\varpi_0} z_0,\dots
\widehat{z_1},\dots,\lambda^{\varpi_2} z_2), \lambda \in
\mu_{\varpi_i}.
$$
Moreover, on $\phi_i(\mathcal{U}_i\cap\mathcal{U}_j)\subset
\mathbb{C}^{n}/\mu_{\varpi_i}$ we have the transitions maps
$$
\phi_i\circ
\phi_j^{-1}((z_1,\dots,z_n)_{\varpi_i})=\left(\frac{z_0}{z_j^{\varpi_0/\varpi_j,}},\dots,\frac{\widehat{z_j}}{z_j},
\dots,\frac{1}{z_j^{\varpi_i/\varpi_j}},\dots,\frac{z_n}{z_j^{\varpi_n/\varpi_j}}
\right)_{\varpi_j}
$$
We conclude that $\{\phi_i,\mathcal{U}_i\}_{i=0}^{n}$ is a
holomorphic atlas orbifold for
$\mathbb{P}(\varpi_0,\dots,\varpi_n)$. Moreover, we have that
$\{\mathbb{C}^n,\mu_{\varpi_i},\pi\circ\phi_i \}_{i=0}^{n}$ is an
n-dimensional uniformizing system of
$\mathbb{P}(\varpi_0,\dots,\varpi_n)$.

Since $\varpi_0,\dots,\varpi_n$ are pairwise prime, the singular set
of $\mathbb{P}(\varpi_0,\dots,\varpi_n)$ are the following  $n+1$
points
$$
[1,0,\dots,0],[0,1,\dots,0],\dots,[0,\dots,0,1].
$$

There exist other orbifold structure for
$\mathbb{P}(\varpi_0,\dots,\varpi_n)$. This one, is induced by the
following action of the group
$(\mu_{\varpi_0}\times\cdots\times\mu_{\varpi_n})$ on $\mathbb{P}^n$
given by
$$
\begin{array}{ccc}
    (\mu_{\varpi_0}\times\cdots\times\mu_{\varpi_n})\times \mathbb{P}^n &\longrightarrow&   \mathbb{P}^n \\
    \\
    ((\lambda_0,\dots,\lambda),[z_0,\dots,z_n]) & \longmapsto
    &[\lambda_0z_0,\dots,\lambda_nz_n].
  \end{array}
$$
Now consider the map $\varphi:\mathbb{P}^n\longrightarrow
\mathbb{P}(\varpi_0,\dots,\varpi_n)$  defined by $\varphi([z_0,\dots
,z_n])=[z_0^{\varpi_0},\dots ,z_n^{\varpi_n}]$. We can see that
$\varphi$ induces a homeomorphism
$$\widetilde{\varphi}:\mathbb{P}^n/(\mu_{\varpi_0}\times\cdots\times\mu_{\varpi_n})\longrightarrow
\mathbb{P}(\varpi_0,\dots,\varpi_n).$$ Therefore, we get that
$\mathbb{P}(\varpi_0,\dots,\varpi_n)\simeq\mathbb{P}^n/(\mu_{\varpi_0}\times\cdots\times\mu_{\varpi_n})$
is a orbifold structure given as global quotient.

\begin{obs}\label{grau}
The degree of the map $\varphi:\mathbb{P}^n\longrightarrow
\mathbb{P}(\varpi)$ is equal the order of the group
$(\mu_{\varpi_0}\times\cdots\times\mu_{\varpi_n})$, i.e,
$deg(\varphi)=\varpi_0\cdots \varpi_n$. For details see \cite{Ab}.
\end{obs}

\subsection{Toric structure of  $\mathbb{P}(\varpi_0,\dots,\varpi_n)$}

Let $\varpi=\{\varpi_0,\dots,\varpi_n\}$   be the set of positive
integers satisfying the condition $gcd(\varpi_0,\dots,\varpi_n) = 1$
. Choose $n+1$ vectors $e_0,\dots,e_n$ in a $n$-dimensional real
space $V$ such that $V$ is spanned by $e_0,\dots,e_n$  and there
exists the linear relation
$$
e_0\varpi_0+\dots+e_n\varpi_n=0.
$$
Define the lattice $N=\langle e_0,\dots,e_n\rangle_{\mathbb{Z}}$
consisting of all integral linear combinations of $e_0,\dots,e_n$.
Observe that $N\otimes\mathbb{R}=V$. Let $\Sigma$ be the set of all
possible simplicial cones in V generated by proper subsets of
$e_0,\dots,e_n$. Then $\Sigma(\varpi)$  is a rational simplicial
complete $n$-dimensional fan.

We will show that The toric variety $X_{\Sigma(\varpi)}$ associated
to fan $\Sigma(\varpi)$ is isomorphic to
$\mathbb{P}(\varpi)=\mathbb{P}(\varpi_0,\dots,\varpi_n)$. For this,
we shall use the construction due to D. Cox \cite{C}, of toric
variety as global quotient.

Let $\rho_1,\dots,\rho_r$ the one-dimensional cones of a fan
$\Sigma$ and $\textbf{n}_i\in \mathbb{Z}^n$ denote the primitive
element generator of $\rho_i\cap \mathbb{Z}^n$. Then introduce
variables $z_i$, for $i=1,\dots,r$, for each cone $\sigma\in
\Sigma$. We get the monomial
$$
z^{\widehat{\sigma}}=\prod_{\textbf{n}_i\notin \sigma}z_i
$$
which is the product of all variables not coming form edges of
$\sigma$. Then define $\mathcal{Z}=V(z^{\widehat{\sigma}}; \sigma
\in \Sigma)\subset \mathbb{C}^r$. Now consider the group $G\subset
(\mathbb{C}^*)^r$ given by
$$
G=\{(t_1,\dots,t_r)\in (\mathbb{C}^*)^r;\prod_{i=1}^{r}t_i^{\langle
e_j,\textbf{n}_i\rangle}=1, j=1,\dots, r \}
$$

\begin{teo}\cite{C}\label{cox}
If $X_{\Sigma}$ is a toric variety where
$\textbf{n}_1,\dots,\textbf{n}_r$ span $\mathbb{R}^n$, then:

\begin{enumerate}

  \item [i)]$X_{\Sigma}$ is a universal categorical quotient $(\mathbb{C}^r-{\mathcal{Z}})/G$

  \item [ii)] $X_{\Sigma}$ is a orbifold
  $(\mathbb{C}^r-{\mathcal{Z}})/G$ if and only if $X_{\Sigma}$ is
  simplicial

\end{enumerate}

\end{teo}

Now we return to the weighted projective space. Consider the toric
variety $X_{\Sigma(\varpi)}$ associated to the fan $\Sigma(\varpi)$
with the generators $e_0,\dots,e_n$ of the one-dimensional cone,
with relation $e_0=-\sum_{i=1}^ne_i(\varpi_i/\varpi_0).$ In this
case we see that $\mathcal{Z}=V(z_0,\dots,z_n)=\{0\}$ and the group
$G$ is given by the relations
$$
t_0^{-\frac{\varpi_1}{\varpi_0}}t_1=t_0^{-\frac{\varpi_2}{\varpi_0}}t_2=\cdots=t_0^{-\frac{\varpi_n}{\varpi_0}}t_n=1
$$
If we shall take $t=t_0^{\frac{1}{\varpi_0}}$ we get
$G(\varpi)=\{(t^{\varpi_0},\dots,t^{\varpi_n})\in
(\mathbb{C}^*)^{n+1}\}\simeq \mathbb{C}^* $. Therefore, from theorem
\ref{cox} we have that
$$
X_{\Sigma(\varpi)}\simeq (\mathbb{C}^n)/G(\varpi)\simeq
\mathbb{P}(\varpi).
$$
This show that the weighted projective space $\mathbb{P}(\varpi)$ is
a complete simplicial toric variety.

\subsection{$\mathbb{Q}$-line bundles of
$\mathbb{P}(\varpi_0,\dots,\varpi_n)$}

Let $\frac{d}{r}\in \mathbb{Q}$ be, with $gcd(r,d)=1$ and $r>0$.
Consider the $\mathbb{C}^*$-action
$\zeta_{\left(\frac{d}{r}\right)}$ on $\mathbb{C}^{n+1}\backslash
\{0\}\times \mathbb{C}$ given by
$$
\begin{array}{ccc}
   \zeta_{\left(\frac{d}{r}\right)}:\mathbb{C}^*\times\mathbb{C}^{n+1}\backslash
   \{0\}\times \mathbb{C}& \longrightarrow & \mathbb{C}^{n+1}\backslash \{0\}\times \mathbb{C} \\
  (\lambda,(z_0,\dots,z_n),t) & \longmapsto & ((\lambda^{r\varpi_0}
z_0,\dots,\lambda^{r\varpi_n} z_n),\lambda^{-d}t).
\end{array}
$$
We denote  quotient space induced by the action
$\zeta_{\left(\frac{d}{r}\right)}$ by
$$\mathscr{O}_{\mathbb{P}(\varpi)}(d/r):=(\mathbb{C}^{n+1}\backslash
\{0\}\times \mathbb{C})/\sim \zeta_{\left(\frac{d}{r}\right)}.$$ The
space $\mathscr{O}_{\mathbb{P}(\varpi)}(d/r)$ is a line orbibundle
on $\mathbb{P}(\varpi)$. We shall describe the global holomorphic
section of $\mathscr{O}_{\mathbb{P}(\varpi)}(d/r)$, for $d>0$.

\begin{prop}
Let $\mathbb{P}(\varpi):=\mathbb{P}(\varpi_0,\dots,\varpi_n)$, then
$$
H^0(\mathbb{P}(\varpi),\mathscr{O}_{\mathbb{P}(\varpi)}(d/r))=
\bigoplus_{\varpi_0k_0+\cdots+\varpi_nk_n=\frac{d}{r}}\mathbb{C}\cdot(z_0^{k_1}\cdots
z_n^{k_n}).
$$
\end{prop}

\begin{proof}
A global section of this line orbibundle is a linear combination of
the monomials $z^{k}=z_0^{k_1}\cdots z_n^{k_n}$, invariants by the
action $\zeta_{\left(\frac{d}{r}\right)}$, that is,
$\zeta_{\left(\frac{d}{r}\right)}([z,z^{k}])=[z,z^{k}].$ Using this
action we obtain
$$
\begin{array}{lll}
  [(z_0\dots ,z_n),(z_0^{k_1}\cdots z_n^{k_n})] & = & [(\lambda^{r\varpi_0} z_0,\dots,\lambda^{r\varpi_n}
z_n),\lambda^{\sum_{i=0}^{n}r\varpi_ik_i}(z_0^{k_1}\cdots z_n^{k_n})]\\
  & =&[(z_0\dots ,z_n),\lambda^{-d+\sum_{i=0}^{n}r\varpi_ik_i}(z_0^{k_1}\cdots
  z_n^{k_n})].
\end{array}
$$
Therefore $\sum_{i=0}^{n}r\varpi_ik_i=d$, hence the proposition
follows.
\end{proof}

The orbibundles $\mathscr{O}_{\mathbb{P}(\varpi)}(d/r)$ can
therefore be considered as elements of the rational Picard group of
$\mathbb{P}(\varpi)$, that is, as $\mathbb{Q}$-line bundles. Is
possible to show that the $\mathbb{Q}$-Picard group can be generated
by $\mathscr{O}_{\mathbb{P}(\varpi)}(\varpi_0\dots\varpi_n)$, that
is
$$
Pic(X)\otimes\mathbb{Q} := Pic(X)_{\mathbb{Q}}=\mathbb{Q}\cdot
 \mathscr{O}_{\mathbb{P}(\varpi)}(\varpi_0\dots\varpi_n).
$$

\begin{obs}
It is possible to show that
$\mathscr{O}_{\mathbb{P}(\varpi)}(1)=\varphi^*(\mathscr{O}_{\mathbb{P}^n}(1))$,
where $\mathscr{O}_{\mathbb{P}^n}(1)$ is a line bundle on $\p$. We
Recall that nor always the pull-back of orbibundle is defined. For
this to be possible the map, in question between orbifold, must to
satisfies a certain condition of "goodness".
\end{obs}

There exist a exact sequence of orbibundle on $\mathbb{P}(\varpi)$
similar the Euler's seguence on $\p.$
\begin{teo}\cite{M}
Then there exist a exact sequence given by
$$
0\longrightarrow
\underline{\mathbb{C}}\stackrel{\varsigma}{\longrightarrow}\bigoplus_{i=0}^{n}\mathscr{O}_{\mathbb{P}(\varpi)}(\varpi_i)
\longrightarrow T\mathbb{P}(\varpi) \longrightarrow 0
$$
where $\underline{\mathbb{C}}=\mathbb{P}(\varpi)\times\mathbb{C}$ is
the trivial line orbibundle on $\mathbb{P}(\varpi)$.
\end{teo}
Also, we shall call this sequence as Euler's sequence. The map
$\varsigma$ is given by $
 \varsigma(1)=(\varpi_0z_0,\dots,\varpi_nz_n).$

\begin{defi}
Let $X$ be a $n$-dimensional compact complex orbifold with
uniformizing system $\{(\mathcal{U}_i,G_i,\pi_i)\}_{i\in \Lambda}$
and $\omega\in \Omega_{X}^n$ a $n$-form. The orbifold integral of
$\omega $ on $X$ is defined by
$$
\displaystyle\int_{X}^{orb} \omega= \sum_{i\in
\Lambda}\frac{1}{|G_i|}\displaystyle\int_{\mathcal{U}_i}\pi_i^*\omega,
$$
where $|G_i|$ is the order of the group $G_i$.
\end{defi}

\begin{obs}
Let $Ker(X)=\{g\in \coprod_{i\in \Lambda} G_i; g(x)=x, \forall \
x\in X \}$ and $X_{reg}=X\backslash\{sing(X)\}$. Then
$$
\displaystyle\int_{X}^{orb}
\omega=\frac{1}{\#Ker(X)}\displaystyle\int_{X_{reg}}\omega
$$
See \cite{M}.
\end{obs}

\begin{prop}\label{prop1}
Let $\mathscr{O}_{\mathbb{P}(\varpi)}(1)$ be the hyperplane bundle
on $\mathbb{P}(\varpi)$, then
$$
\displaystyle\int_{\mathbb{P}(\varpi)}^{orb}
c_1(\mathscr{O}_{\mathbb{P}(\varpi)}(1))^n=\dfrac{1}{\varpi_0\dots\varpi_n}.
$$
\end{prop}

\begin{proof}
From the definition of orbifold integral we have
$$
\displaystyle\int_{\mathbb{P}(\varpi)}^{orb}
c_1(\mathscr{O}_{\mathbb{P}(\varpi)}(1))^n=\dfrac{1}{\#Ker(\mathbb{P}(\varpi))}
\displaystyle\int_{\mathbb{P}(\varpi)_{reg}}
c_1(\mathscr{O}_{\mathbb{P}}(\varpi)(1))^n.$$ Since
$\mathbb{P}(\varpi_0,\dots,\varpi_n)\simeq\mathbb{P}^n/(\mu_{\varpi_0}\times\cdots\times\mu_{\varpi_n})$
we conclude that
$$Ker(\mathbb{P}(\varpi))=\bigcap_{i=0}^{n}\mu_{\varpi_i}=\{1\},$$
hence $\#Ker(\mathbb{P}(\varpi))=1$. On the other hand, since
$\varphi^*(\mathscr{O}_{\mathbb{P}(\varpi)}(1))=\mathscr{O}_{\mathbb{P}^n}(1)$
we get
$$
\displaystyle\int_{\mathbb{P}(\varpi)}^{orb}
c_1(\mathscr{O}_{\mathbb{P}(\varpi}(1))^n=
\displaystyle\int_{\mathbb{P}(\varpi)_{reg}}
c_1(\mathscr{O}_{\mathbb{P}}(\varpi)(1))^n=\frac{1}{deg(\varphi)}\displaystyle\int_{\mathbb{P}^n}
c_1(\mathscr{O}_{\mathbb{P}^n}(1))^n=\dfrac{1}{\varpi_0\dots\varpi_n}.
$$
The last equality follows from proposition \ref{grau}.
\end{proof}

\section{$\mathbb{Q}$-line bundle on simplicial toric variety and intersection number}

Let $X$ be a normal toric variety. Since a Weil divisor is a cycle
in $X$ of real dimension $2n-2$, we have a homomorphism
$$
\vartheta:\mathcal{W}(X)\longrightarrow H_{2n-2}(X,\mathbb{Z})
$$
which associates to each Weil divisor it's homology class. In other
hand, there exist (see \cite{H}) an isomorphism
$$
\alpha:\mathcal{C}(X)\stackrel{\simeq}{\longrightarrow} Pic(X)
$$
between the group of classes of Cartier divisors and the Picard
group. This one is the group of isomorphy classes of line bundles
(or isomorphy classes of invertible sheaves) on X. By composition of
$\alpha$ with the morphism $c_1:Pic(X)\rightarrow H^2(X)$, we obtain
a morphism denoted
$$
c_1:\mathcal{C}(X)\longrightarrow H^2(X).
$$
When $X$ is smooth we have that $c_1(\mathscr{O}(D))$ is the
Poincaré's dual of the cycle represented by $D\in \mathcal{C}(X)$.
In general case, we cannot guarantee this, but we will see that is
true if $D$ is a divisors invariants by torus action.

Let $\mathbb{T}$ be the torus which acts in $X$. Denote by
$\mathcal{C}^{\mathbb{T}}(X)$  and $\mathcal{W}^{\mathbb{T}}(X)$,
respectively, the groups of $\mathbb{T}$-invariants divisors of
Cartier and Weil, modulo equivalence of principal
$\mathbb{T}$-invariants divisors .

\begin{teo}\label{brasselet}\cite{BBFK}
Let $X$ be a compact toric variety, there exist a commutative
diagram
$$
\begin{array}{ccc}
  \mathcal{C}^{\mathbb{T}}(X) & \hookrightarrow & \mathcal{W}^{\mathbb{T}}(X) \\
  \downarrow_{\simeq} &  &  \downarrow_{\simeq}\\
  H^2(X,\mathbb{Z}) & \stackrel{\frown[X]}{\longrightarrow} & H_{2n-2}(X,\mathbb{Z})
\end{array}
$$
where the vertical isomorphism correspond to the morphisms $c_1$ and
$\vartheta$.
\end{teo}

When $X$ is simplicial we have  $ Pic(X)_{\mathbb{Q}}\simeq
\mathcal{C}^{\mathbb{T}}(X)=\mathcal{W}^{\mathbb{T}}(X)$. Using
these identifications and temporizing the diagram of the theorem
\ref{brasselet} by $\mathbb{Q}$ we have
$$
\begin{array}{ccc}
  Pic(X)\otimes\mathbb{Q} & \stackrel{\simeq}{\longrightarrow}& \mathcal{C}^{\mathbb{T}}(X)\otimes\mathbb{Q} \\
  \downarrow_{\simeq} &  &  \downarrow_{\simeq}\\
  H^2(X,\mathbb{Q}) & \stackrel{\frown[X]}{\longrightarrow} & H_{2n-2}(X,\mathbb{Q})
\end{array}
$$

Let $X$ be a complete simplicial toric variety of two dimension and
let $D$ be a $\mathbb{Q}$-Cartier divisor on $X$ . Then from theorem
\ref{brasselet} we conclude that $c_1(\mathscr{O}(D))$ is the
Poincaré's dual of the cycle represented by $D$. Therefore, we can
extend the intersection theory for rational coefficients. For
instance, let $D_1,D_2\in\mathcal{W}^{\mathbb{T}}(X)$ the
intersection number is the rational number
$$
D_1\cdot D_2=\langle c_1(\mathscr{O}(D_1))\cap
c_1(\mathscr{O}(D_2)),[X]\rangle\in \mathbb{Q}
$$
as well as in the case with integer coefficients.

We will use the  Poincaré-Satake duality to express the number of
intersection in terms of the orbifold integral.

\begin{prop}\label{prop2}
Let $X$ be a simplicial compact toric variety and $L_1,L_2\in
Pic(X)\otimes\mathbb{Q}$. Then
$$
L_1\cdot L_2=\displaystyle\int_{X}^{orb}c_1(L_1)\wedge c_1(L_2).
$$
\end{prop}

\begin{proof}
Since $X$ be a simplicial compact toric variety, then follows from
theorem \ref{cox} that $X$ is a compact complex orbifold. Let
$H^{i}_{DR}(X)$ be the cohomology group of $i$-forms on $X$ (in
orbifold's sence). We have the following Poincaré's duality for
orbifold showed by Satake in \cite{S} which can be read as
$$
\begin{array}{ccc}
  H^{i}_{DR}(X)\otimes H^{n-i}_{DR}(X)&\longrightarrow &\mathbb{Q} \\
   \alpha\wedge\eta&\longmapsto&\displaystyle\int_{X}^{orb}\alpha\wedge\eta.
\end{array}
$$
It follows Form this duality that
$$
L_1\cdot L_2=\langle c_1(\mathscr{O}(L_1))\cap
c_1(\mathscr{O}(L_2)),[X]\rangle=\displaystyle\int_{X}^{orb}c_1(L_1)\wedge
c_1(L_2).
$$
\end{proof}
Therefore, if $D_1,D_2\in\mathcal{W}^{\mathbb{T}}(X)$  we shall have
the following formula
$$
D_1\cdot D_2= \displaystyle\int_{X}^{orb}c_1(\mathscr{O}(D_1))\wedge
c_1(\mathscr{O}(D_2)).
$$

In general, if $X$ is a compact orbifold and $L_1,L_2\in
Pic(X)\otimes\mathbb{Q}$ we can define the intersection number using
this formula.

\begin{exe}
Let $D_1\in H^0(
\mathbb{P}(\varpi_0,\varpi_1,\varpi_2),\mathscr{O}(d_1))$ and
$D_2\in H^0(
\mathbb{P}(\varpi_0,\varpi_1,\varpi_2),\mathscr{O}(d_2))$ be.
Follows from proposition \ref{prop1} and \ref{prop2} that
$$
D_1\cdot D_2= \displaystyle\int_{X}^{orb}c_1(\mathscr{O}(d_1))\wedge
c_1(\mathscr{O}(d_2))=\displaystyle\int_{X}^{orb}(d_1d_2)\cdot
c_1(\mathscr{O}(1))^2=\frac{d_1d_2}{\varpi_0\varpi_1\varpi_2}.
$$
\end{exe}

\begin{obs}
Let $\{\mathbb{B}^2(0,\epsilon_p),G_p,\pi_p \}$ be  an uniformizing
system of $X$ and $D_1,D_2\in \mathcal{C}(X)\otimes\mathbb{Q}$
$\mathbb{Q}$-Cartier divisors that have no common component. Let
$x\in D_1\cap D_2$, $f$ and $g$ the lifting of the local equations
for $D_1$ and $D_2$, respectively. The orbifold intersection
multiplicities is defined
$$
i_{x}^{orb}(D_1,D_2)=\frac{1}{|G_x|}\cdot
dim_{\mathbb{C}}\frac{\mathcal{O}_0^{G_x}}{\langle f,g\rangle},
$$
where $\mathcal{O}_0^{G_x}\subset \mathcal{O}_0$ is the sub-algebra
local of the functions $G_x$-invariant and $|G_x|$ is the order of
the group $G_p$. With the same arguments used  in the case smooth we
can show that
$$
D_1\cdot D_2=\sum_{x\in D_1\cap D_2}i_{x}^{orb}(D_1,D_2).
$$
\end{obs}

\section{Foliations on cyclic orbifold surfaces }

This section follows the ideas of M. Brunella in \cite{B2}. Let $X$
be complex surface with isolated cyclic quotient singularities
(orbifold), i.e, around each $p\in Sing(X)$ the surface is of the
type $\mathbb{B}^2(0,\epsilon_{p})/\mathcal{G}_{p}^{k}$ , where
$\mathbb{B}^2(0,\epsilon_{p})$ is a ball in $\mathbb{C}^2$ and
$\mathcal{G}_{p}^{k}$ is the cyclic group of order $k$ generated by
$$
\gamma_{p}^{k}(z,w) =(e^{\frac{2\pi i}{k}}z,e^{\frac{2\pi i}{k}h}w)
$$
for suitable coprime positive integers $k, h$ with $0 < h < k$. Let
 $
Sing(X)=\{p_1,\dots,p_s\}$ be and $k_i$ the order of $p_i$,
respectively.

\begin{exe}
Let  $\mathbb{P}(\varpi_0,\varpi_1,\varpi_2)$ be the weighted
projective space and $\mu_{\varpi_i}\subset \mathbb{C}^*$ the
subgroup of $\varpi_i$-th roots of unity. We saw in the section
 that $\{\mathbb{C}^2,\mu_{\varpi_i},\pi\circ\phi_i \}_{i=0}^{2}$ is an
2-dimensional uniformizing system of
$\mathbb{P}(\varpi_0,\varpi_1,\varpi_2)$. That is, the weighted
projective plane $\mathbb{P}(\varpi_0,\varpi_1,\varpi_2)$ is a
cyclic orbifold surface.
\end{exe}

A foliation $\F$ on X is given by an open covering
$\{\mathcal{U}_i\}_{i\in \Lambda}$ of X and holomorphic vector
fields $\vartheta_i\in H^0(\mathcal{U}_i,TX_{|\mathcal{U}_i})$ with
isolated zeroes such that
 $$Sing(\vartheta_i)\cap Sing(X)=\emptyset$$
 and
$\vartheta_i=g_{ij}\cdot\vartheta_j$  on
 $\mathcal{U}_i\cap\mathcal{U}_j$
for some non-vanishing holomorphic functions $g_{ij}\in
\mathcal{O}^*(\mathcal{U}_i\cap\mathcal{U}_j)$ . Therefore, the
singular set of $\F$ is the discrete subset of
$X\backslash\{Sing(X)\}$ defined by
$$
Sing(\F)=\bigcup_{i\in \Lambda} Sing(\vartheta_i).
$$

Since $Sing(\F)\cap Sing(X)=\emptyset$, the leaf of $\F$ through
$p_i\in Sing(X)$ is an orbifold in which $p_i$ is affected by the
multiplicity $k_i$. This local leaf on
$\mathbb{B}^2(0,\epsilon_{i})/\mathcal{G}_{i}^{k_i}$ is of the form
$\mathbb{D}(0,\delta_i)/\mathcal{G}_{i}^{k_i}$, where
$\mathbb{D}(0,\delta_i)\subset \mathbb{C}$ is a disk. In other
words, on $\mathbb{B}^2(0,\epsilon_{i})/\mathcal{G}_{i}^{k_i}$ the
foliation $\F$ is the quotient of the vertical or horizontal
foliation on $\mathbb{B}^2(0,\epsilon_{i})$, up to an equivariant
biholomorphism.

The functions $g_{ij}$ form a multiplicative cocycle and define a
holomorphic $\mathbb{Q}$-line bundle $K_{\F}$ on $X$, called
canonical bundle of $\F$. In fact, let $K_{\F}^{reg}$ be the
canonical bundle(or sheave) of $\F_{|_{X_{reg}}}$, where
$\F_{|_{X_{reg}}}$ is the restriction of foliation to smooth part of
$X$. Take the direct image of $K_{\F}^{reg}$ under the inclusion
$$i:X_{reg}=X \backslash Sing(X)\rightarrow X.$$
Hence we have the sheave $i_{*}K_{\F}^{reg}=K_{\F}$ on $X$ which is
not locally free at $p_i\in Sing(X)$, but its $\kappa$-power
$K_{\F}^{\otimes \kappa}$  is, where $\kappa=l.c.m(k_1,\dots,k_s)$,
i.e, $K_{\F}^{\otimes \kappa}$ is a line bundle. The vector field
$\frac{\partial}{\partial w}$ on $\mathbb{B}^2(0,\epsilon_{i})$ is
not $\mathcal{G}_{i}^{k_i}$-invariant but
$\left(\frac{\partial}{\partial w}\right)^{\otimes \kappa}$ is.
Indeed, since
$\mathcal{G}_{i}^{k_i}=\langle\gamma_{p_j}^{k_j}\rangle$ we have
$$
\gamma_{p_j}^{k_j}\cdot\left(\frac{\partial}{\partial
w}\right)^{\otimes \kappa}=\left(e^{\frac{2\pi
i}{k_j}h_j}\frac{\partial}{\partial
w}\right)\otimes\cdots\otimes\left(e^{\frac{2\pi
i}{k_j}h_j}\frac{\partial}{\partial w}\right)=e^{2\pi
i\widehat{k_j}h_j}\left(\frac{\partial}{\partial w}\right)^{\otimes
\kappa}=\left(\frac{\partial}{\partial w}\right)^{\otimes \kappa},
$$
where $\widehat{k_j}=l.c.m(k_1,\dots,k_s)/k_j\in \mathbb{Z}_{+}$.
Moreover, the relations $\vartheta_i^{\otimes \kappa}=g_{ij}^{
\kappa}\cdot\vartheta_j^{\otimes \kappa}$ allow us to construct a
global holomorphic section of line bundle $(TX\otimes
K_{\F})^{\otimes \kappa}$, i.e, a foliation $\F$ on $X$ is a global
section of $\mathbb{Q}$-line bundle $TX\otimes K_{\F}$. Therefore
the space of holomorphic foliations on $X$ is given by
$$H^0(X,TX\otimes K_{\F})$$

Also, we can define foliations using $1$-forms. That is, a foliation
can be defined by a collection of $1$-forms $\omega_j\in
\Omega_{X}^1(\mathcal{U}_j)$ with isolated zeros and such that

\begin{center}
$\omega_i=f_{ij}\cdot \omega_j$\ \ \  on
$\mathcal{U}_i\cap\mathcal{U}_j$,\ \ \  $f_{ij}\in
\mathcal{O}_{X}^*(\mathcal{U}_i\cap\mathcal{U}_j)$.
\end{center}
Again, the functions $\{f_{ij}\}$ defines a $\mathbb{Q}$-line bundle
$N_{\F}$ on $X$ called normal bundle of $\F$. By contraction to
$2$-forms we have that $K_{X}\simeq Hom(K_{\F}^*,N_{\F})\simeq
K_{\F}\otimes N_{\F}^* $.

Let $S\subset X$ be a compact connected (possibly singular) curve,
and suppose that each irreducible component of C is not invariant by
F. For every $p\in S$ we can define an index $tang(\F,S, p)$ which
measure the tangency order of $\F$ with $S$ at $p$.

Take a neighborhood
$\mathbb{B}^2(0,\epsilon_{i})/\mathcal{G}_{i}^{k_i}$ and we lift
$\F_{|\mathbb{B}^2(0,\epsilon_{i})/\mathcal{G}_{i}^{k_i}}$ and
$S\cap \mathbb{B}^2(0,\epsilon_{i})/\mathcal{G}_{i}^{k_i}$ on
$\mathbb{B}^2(0,\epsilon_{i})$. Let $\vartheta$ be and $f$,
respectively, the vector field
 and the local equation on $\mathbb{B}^2(0,\epsilon_{i})$ that define the lifting of $\F$ and $S$. Then
 we define the index by
$$tang(\F,S, p_i) =\frac{1}{k_i}\cdot
dim_{\mathbb{C}}
\frac{\mathcal{O}_0^{\mathcal{G}_{i}^{k_i}}}{\langle f, \vartheta(f)
\rangle}
$$
where $\mathcal{O}_0^{\mathcal{G}_{i}^{k_i}}$ denotes the local
algebra of germs of functions $ \mathcal{G}_{i}^{k_i}$-invariants on
$\mathbb{B}^2(0,\epsilon_{i})$. We then have the following formula
$$
tang(\F,S)=K_{\F}\cdot S + S\cdot S,
$$
where $tang(\F,S)=\sum_{p\in S}tang(\F,S, p_i)$. See \cite{B1} and
\cite{B2}.

\subsection{ The orbifold Milnor number }

Let $\F$ be a foliation on $X$ and $p\in Sing(\F)$.Take an
uniformized chart of $p$ given by
$\{\mathbb{B}^2(0,\epsilon_p),G_p,\pi_p \}$. Let
$V=P\frac{\partial}{\partial x}+Q\frac{\partial}{\partial x}$ be the
vector field $G_p$-invariant on $\mathbb{B}^2(0,\epsilon_p)$ which
induce the lifting $\pi_p^*(\F)$ of the foliation $\F$. The
\emph{Orbifold Milnor's number} is defined by
$$
\mu_{p}^{orb}(\F)=\frac{1}{|G_p|}\cdot
dim_{\mathbb{C}}\frac{\mathcal{O}_0^{G_p}}{\langle P,Q\rangle},
$$
\\
where $|G_p|$ is the order of the group $G_p$. We say that $p\in
Sing(\F)$ is \emph{non-degenerated} if $\mu_{p}^{orb}(\F)=1$, i.e,
when $dim_{\mathbb{C}}\frac{\mathcal{O}_0^{G_p}}{\langle
P,Q\rangle}=|G_p|.$

\begin{prop}
Let $\F$ be a holomorphic foliation on $X$, then
$$
\sum_{x\in Sin(\F)}\mu_p^{orb}(\F)=K_{\F}\cdot K_{\F}+ K_{\F}\cdot
K_{X}+ \chi_{top}(X)-\sum_{p\in
Sing(X)}\left(1-\frac{1}{|G_p|}\right)
$$
\end{prop}

\begin{proof}
The proof  follows in the same lines in the case smooth, and this
one is proved in \cite{B1}. That is, we must to calculate the Chern
class $c_2(TX\otimes K_{\F})$. Therefore, we get
$$
\sum_{x\in Sin(\F)}\mu_p^{orb}(\F)=K_{\F}\cdot K_{\F}+ K_{\F}\cdot
K_{X}+ c_2(TX\otimes K_{\F}).
$$
On the other hand, $c_2(TX\otimes
K_{\F})=\chi_{orb}(X)=\chi_{top}(X)-\sum_{p\in
Sing(X)}\left(1-\frac{1}{|G_p|}\right)$.
\end{proof}

\section{Foliations on weighted projective space}
A vector fields on $\mathbb{C}^{n+1}$ is called quasi-homogeneous of
type $(\varpi_0,\ldots,\varpi_n)$ and degree  $d$ if is of the form
$$
X=\sum_{i=0}^{n}P_i\dfrac{\partial}{\partial z_i},
$$
where
$P_i(\lambda^{\varpi_0}z_0,\ldots,\lambda^{\varpi_n}z_n)=\lambda^{\varpi_i+d-1}P_i(z_0,\ldots,z_2)$,
for all $i=0,\dots,n.$ The foliation $\F_{X}$ induced by these
vector fields is called quasi-homogeneous foliations. The linear
vector fields
$$
\mathcal{R}_{\varpi}=\sum_{i=0}^{n}\varpi_iz_i\dfrac{\partial}{\partial
z_i}
$$
is the infinitesimal generator of the $\mathbb{C}^*$-action on
$\mathbb{C}^{n+1}\backslash\{0\}$ given by
$\lambda\cdot(z_0,\cdots,z_n)=(\lambda^{\varpi_0}z_0,\ldots,\lambda^{\varpi_n}z_n)$.
Any quasi-homogeneous foliation $\F_{X}$ of type
$(\varpi_0,\ldots,\varpi_n)$ is preserved by this action. We shall
explain this. If $X$ has degree $d$ we have
$$\lambda\cdot X(z_0,\ldots,z_2)= \lambda^{1-d}X(\lambda\cdot(z_0,\ldots,z_n)),$$
that is, if $\mathcal{L}_p$ is the  leaf of $\F_{X}$ through p then
$\lambda\cdot \mathcal{L}_p= \mathcal{L}_{\lambda\cdot p}$. This
means that the flow of vector field $\mathcal{R}_{\varpi}$ takes
leaves of $\F_{X}$ onto leaves of $\F_{X}$, moreover, as you can see
$\lambda\cdot Sin(X)=Sing(X)$. Then we can conclude that
$\lambda\cdot \F_{X}=\F_{X}$. Observe that when $d=1$ occurs
$\lambda\cdot \mathcal{L}_p= \mathcal{L}_{p}$.

This say us that a quasi-homogeneous foliation $\F_{X}$ of type
$(\varpi_0,\ldots,\varpi_n)$ can be  projected on the foliation on
$\mathbb{P}(\varpi_0,\ldots,\varpi_n)$ via the projection map
$\pi:\mathbb{C}^3\backslash\{0\} \rightarrow
\mathbb{P}(\varpi_0,\ldots,\varpi_n)$. We shall see that all
foliations on $\mathbb{P}(\varpi_0,\ldots,\varpi_n)$ can be
represented in "homogeneous" coordinates by a vector field
quasi-homogeneous of type $(\varpi_0,\ldots,\varpi_n)$ and some
degree $d$.

\begin{defi}
A holomorphic foliation $\F$ on $\mathbb{P}(\varpi)$ of degree $d$
is a global holomorphic section of line orbibundle
$T\mathbb{P}(\varpi)\otimes \mathscr{O}(d-1)$.
\end{defi}
Tensorizing the Euler's sequence by
$\mathscr{O}_{\mathbb{P}(\varpi)}(d-1)$ we shall have
$$
0\longrightarrow
\mathscr{O}_{\mathbb{P}(\varpi)}(d-1)\longrightarrow
\bigoplus_{i=0}^{n}\mathscr{O}_{\mathbb{P}(\varpi)}(d+\varpi_i-1)
\longrightarrow T\mathbb{P}(\varpi)\otimes \mathscr{O}(d-1),
\longrightarrow 0
$$
therefore the space of holomorphic foliations on
$\mathbb{P}(\varpi)$ is
$$\mathbb{P}H^0(\mathbb{P}(\varpi),T\mathbb{P}(\varpi)\otimes
\mathscr{O}_{\mathbb{P}(\varpi)}(d-1))\simeq
\mathbb{P}H^0\left(\mathbb{P}(\varpi),\frac{\bigoplus_{i=0}^{n}\mathscr{O}_{\mathbb{P}(\varpi)}(d+\varpi_i-1)}
{\mathcal{R}_\varpi\cdot\mathscr{O}_{\mathbb{P}(\varpi)}(d-1)}\right).
$$
Then a foliation of degree $d$ on $\mathbb{P}(\varpi)$ can be
represented in homogeneous coordinates by a vector field
quasi-homogeneous of type $(\varpi_0,\ldots,\varpi_n)$ and degree
$d$, modulo addition of a vector field of the form
$g\cdot\mathcal{R}_\varpi$, where $g$ is a quasi-homogeneous
polynomial of degree $d-1$.

The degree of a foliation $\F$ on
$\mathbb{P}(\varpi_0,\varpi_0,\varpi_2)$ is a number of tangency of
$\F$ with a generic element of the linear system $\ell\in
H^0(\mathbb{P}(\varpi_0,\varpi_0,\varpi_2),\mathscr{O}(\varpi_0,\varpi_0,\varpi_2))$.
In fact, we have that
$$
Tang(\F,h)=K_{\F}\cdot \ell+\ell\cdot
\ell=\frac{(\varpi_0\varpi_0\varpi_2)(d-1)}{\varpi_0\varpi_0\varpi_2}+\frac{\varpi_0\varpi_0\varpi_2}{\varpi_0\varpi_0\varpi_2}=d.
$$

\subsection{Foliations on  $\mathbb{P}(\varpi_0,\varpi_1,\varpi_2)$
given by $1$-forms}

Let $|\varpi|=\varpi_0+\varpi_1+\varpi_2$ be. A foliation $\F$ on
$\mathbb{P}(\varpi)$ can be given by a section $\Omega\in
\Omega_{\mathbb{P}(\varpi)}^1\otimes N_{\F}$. Since
$K_{\mathbb{P}(\varpi)}=K_{\F}\otimes N^*_{\F}$,
$K_{\F}=\mathcal{O}_{\mathbb{P}(\varpi)}(d-1)$ and
$K_{\mathbb{P}(\varpi)}=\mathcal{O}_{\mathbb{P}(\varpi)}(-|\varpi|)$,
we have that
$N_{\F}=\mathcal{O}_{\mathbb{P}(\varpi)}(d+|\varpi|-1)$. Now,
tensorizing the dual Euler's sequence by
$\mathcal{O}_{\mathbb{P}(\varpi)}(d+|\varpi|-1)$ we get the exact
sequence
$$
0\rightarrow
\Omega_{\mathbb{P}(\varpi)}^1\otimes\mathcal{O}_{\mathbb{P}(\varpi)}(d+|\varpi|-1)
\rightarrow\bigoplus_{i=0}^{n}\mathscr{O}_{\mathbb{P}(\varpi)}(d+|\varpi|-\varpi_i-1)\stackrel{i_{\mathcal{R}_{\varpi}}}{\rightarrow}
 \mathcal{O}_{\mathbb{P}(\varpi)}(d+|\varpi|-1)\rightarrow 0.
$$
Therefore a foliation can be given in homogeneous coordinate by
$1$-form $$\Omega=A_0dz_0+A_1dz_1+A_2dz_2,$$ where $A_i$ is a
quasi-homogeneous polynomial of the type
$(\varpi_0,\varpi_1,\varpi_2)$ and degree $d+|\varpi|-\varpi_i-1,$
and $i_{\mathcal{R}_{\varpi}}(\Omega)=0$, i.e,
$\varpi_0A_0z_0+\varpi_1A_1z_1+\varpi_2A_2z_2=0$.

\begin{exe}Let $f$ and $g$ quasi-homogeneous polynomial of the type
$(\varpi_0,\varpi_1,\varpi_2)$ and degree $d_1$ and $d_2$,
respectively. The $1$-form
 $\Omega(f,g)=d_1fdg-d_2gdf$ defines a foliation $\F$ on
 $\mathbb{P}(\varpi)$ of degree $d_1+d_2-|\varpi|$. Moreover, the
 rational function $f^{d_1}/g^{d_2}$ is a first integral for $\F$.
\end{exe}

\begin{exe}(logarithmic foliations )
\\
Let $f_1,\dots, f_k$ quasi-homogeneous polynomial of the type
$(\varpi_0,\varpi_1,\varpi_2)$ and degrees $d_1,\dots,d_k$,
respectively, with $k\geq3$. Let $\lambda_1,\dots,\lambda_k\in
\mathbb{C}^*$ be,  such that $\sum_{=1}^{k}\lambda_id_i=0$. Define
the $1$-form given by
$$
\Omega=(f_1\cdots f_k)\cdot\sum_{=1}^{k}\lambda_i\frac{df_i}{f_i}.
$$
Follows from the Euler's formula that
$i_{\mathcal{R}_{\varpi}}(\Omega)=(f_1\cdots
f_k)\cdot\left(\sum_{=1}^{k}\lambda_id_i\right)=0$. Therefore,
$\Omega$ define a foliation on $\mathbb{P}(\varpi)$ of degree
$\sum_{=1}^{k}d_i-|\varpi|$.
\end{exe}

\section{The number of singularities with multiplicities}

We have the following formula for the number of singularities for a
foliation on $\mathbb{P}(\varpi_0,\dots,\varpi_n)$.

\begin{teo}\label{1}
Let $\F$ a foliation on  $\mathbb{P}(\varpi_0,\dots,\varpi_n)$ with
isolated singularities. Then
$$
(\varpi_0,\cdots,\varpi_n)\cdot\sum_{p\in
Sin(\F)}\mu_p^{orb}(\F)=\sum_{i=0}^{n}\left[\sum_{k=0}^{i}(-1)^{i-k}\sigma_{n-i}(\varpi_0,\dots,\varpi_n)d^k\right],
$$
where $\sigma_j$ is the  $j$-th elementary symmetric function.
\end{teo}

\begin{proof}
From the Euler's sequence we have that
$$c(T\mathbb{P}(\varpi_0,\dots,\varpi_n))=\prod_{i=0}^n(1+c_1(\mathscr{O}_{\mathbb{P}(\varpi)}(\varpi_i))).$$
Let $c_1(\mathscr{O}_{\mathbb{P}(\varpi)}(1))=h$ be, then $
c_i(\mathbb{P}(\varpi_0,\varpi_1,\varpi_2)=\sigma_i(\varpi_0,\varpi_1,\varpi_2)\cdot
h^i, $ where $\sigma_i$ is the  $i$-th elementary symmetric
function. Let $\vartheta \in
T\mathbb{P}(\varpi_0,\dots,\varpi_n)\otimes\mathscr{O}(d-1)$ be. It
follows  from intersection theory \cite{F} that the Chern class
$c_{n}(T\mathbb{P}(\varpi_0,\dots,\varpi_n)\otimes\mathscr{O}(d-1))$
is the intersection product between the graph of $\vartheta$ and the
graph of the null section, and each singularity $p\in Sing(\F)$
gives a contribution equal to $\mu_p^{orb}(\F)$. Hence, we get
$$
\begin{array}{ccl}
  \displaystyle\sum_{p\in Sin(\F)}\mu_p^{orb}(\F)&=&\displaystyle\int\limits_{\mathbb{P}(\varpi_0,\dots,\varpi_n)}^{orb}c_{n}(T\mathbb{P}(\varpi_0,\dots,\varpi_n)\otimes
  \mathscr{O}(d-1))\\
  \\
   & = &\dfrac{\sum_{i=0}^{n}\left[\sum_{k=0}^{i}(-1)^{i-k}\sigma_{n-i}(\varpi_0,\dots,\varpi_n)d^k\right]}{\varpi_0\cdots\varpi_n}.
\end{array}
$$

\end{proof}

\section{ Extatic hypersurface}

Let $F,G\in
H^0(\mathbb{P}(\varpi),\mathscr{O}_{\mathbb{P}(\varpi)}(k))$ be. We
have that $\Theta(F,G)=\frac{F}{G}$ is a well defined rational
function on $\mathbb{P}(\varpi)$, i.e, it defines a rational
function $\Theta(F,G):
\mathbb{P}(\varpi_0,\dots,\varpi_n)\dashrightarrow \mathbb{P}$. We
say that a foliation $\F$ on $\mathbb{P}(\varpi)$ admits a rational
first integral of degree $k$ if there exist $F,G\in
H^0(\mathbb{P}(\varpi),\mathscr{O}_{\mathbb{P}(\varpi)}(k))$ such
that $X(\Theta(F,G))=0$, where $X$ is a vector fields that defines
$\F$ in homogeneous coordinates.

A finite dimensional linear system $V$ on
 $\mathbb{P}(\varpi)$ is the same as to consider a finite dimensional linear space of
quasi-homogeneous polynomials $V$ in the variable $z =
(z_0,\dots,z_n)$. Suppose now that $V$ is a finite dimensional
linear system and let $v_1,\dots,v_{\ell}\in
\mathbb{C}[z_0,\dots,z_n]$ be a basis of $V$ . Consider the matrix
 $$
  E(V,X)=\left(\begin{array}{cccc}
  v_1& v_2 & \cdots & v_{\ell}\\
  \\
   X(v_1) &X(v_2) & \cdots & X(v_{\ell})\\
   \\
  \vdots & \vdots & \ddots & \vdots\\
  \\
X^{\ell -1}(v_1) &X^{\ell -1}(v_2) & \cdots & X^{\ell -1}(v_{\ell})
\end{array}%
\right)
 .$$

The \emph{extactic of} $X$ \emph{associated to} $V$ is
$\mathcal{E}(V,X) = det\ E(V,X)$, and the \emph{extactic
hypersurface  of} $X$ associated to $V$ is the variety
$Z(\mathcal{E}(V,X))$. The concept of \emph{extactic divisor} of  on
a complex manifold and its properties was introduced by J.V.Pereira
[P].

The following result elucidate the role of the extactic variety.

\begin{prop}\label{ext}
Let $\F_{X}$ be a foliation on $\mathbb{P}(\varpi)$ induced in
homogeneous coordinate by a vector fields $X$. Consider a linear
system $V$ on $\mathbb{P}(\varpi)$ and $\{f=0\}$ a
$\F_{X}$-invariant hypersurface with $f\in V$. Then $f$ is a factor
of $\mathcal{E}(V,X)$. Moreover, $\F$  admit a rational integral
first if, and only if,  $\mathcal{E}(V,X)=0$.

\end{prop}

\begin{proof}
This proposition it follows  using the same ideas of Theorem 4.3 of
\cite{C-L-P}, for the case of polynomial vector fields on
$\mathbb{C}^2.$
\end{proof}
If $f$ is a defining equation for an irreducible $\F_{X}$-invariant
hypersurface, such that $f\in V$ , its \emph{multiplicity} is
defined by the largest integer $m$ such that $f^{m}$ divides
$\mathcal{E}(V,X)$.

\section{Proofs}

\subsection{Proof of theorem \ref{2}}
 Let $X$ be a vector fields on $\mathbb{C}^{n+1}$ that defines $\F$
in homogeneous coordinates and
$V=H^0(\mathbb{P}(\varpi),\mathscr{O}_{\mathbb{P}(\varpi)}(k))$.
Since $\F$ does not admit a rational first integral then by
 proposition \ref{ext} we have that $\mathcal{E}(X,V)\neq 0$. We shall
determinate the degree of $\mathcal{E}(X,V)\neq 0$. Expanding the
determinant we get
$$
\mathcal{E}(X,V)=\displaystyle \sum_{\sigma\in S_{\eta}}
sgn(\sigma)X^{0}(\upsilon_{1\sigma(1)})X(\upsilon_{2\sigma(2)})\cdots
X^{\eta-1}(\upsilon_{\eta\sigma(\eta)}),
$$
where
$\eta=h^0(\mathbb{P}(\varpi),\mathscr{O}_{\mathbb{P}(\varpi)}(k))$
and $\{\upsilon_{i}\}_{i=1}^{\eta}$ is a base for $V$. We have that
$deg(X^{j}(\upsilon_{r}))=j(d-1)+k$, for all $j=0,\dots,\eta$.
Indeed, let $X=\sum_{i=0}^{n}P_i\frac{\partial}{\partial z_i}$ be.
Since $P_i$ is quasi-homogeneous of degree $d+\varpi_i-1$ and
$deg\left(\frac{\partial\upsilon_{r}}{\partial
z_i}\right)=k-\varpi_i$, we have that
$$
deg\left(P_i\cdot\frac{\partial \upsilon_{r}}{\partial
z_i}\right)=deg(P_i)+deg\left(\frac{\partial\upsilon_{r}}{\partial
z_i}\right)=d+\varpi_i-1+k-\varpi_i=d-1+k.
$$ Therefore, inductively
follows that $X^{j}(\upsilon_{r})=j(d-1)+k$. Hence
$$
\begin{array}{lcl}
   deg(E(X,V))& = & \displaystyle \sum_{j=0}^{\eta-1}j(d-1)+k=\displaystyle {\eta\choose
2}(d-1)+\eta\cdot k \\
\\
&=& \displaystyle
{h^0(\mathbb{P}(\varpi),\mathscr{O}_{\mathbb{P}(\varpi)}(k))\choose
2}(d-1)+h^0(\mathbb{P}(\varpi),\mathscr{O}_{\mathbb{P}(\varpi)}(k))\cdot
k.
\end{array}
$$
Let $\mathscr{N}(k)$ be the number of hypersurfaces
$\mathcal{X}$-invariant of degree $k$. We have  $$
k\mathscr{N}(k)\leq deg(E(X,V)).
$$
Hence we get
$$
k\mathscr{N}(k)\leq \deg(E(X,V))=k\cdot
h^0(\mathbb{P}(\varpi),\mathscr{O}_{\mathbb{P}(\varpi)}(k))+(d-1)
\cdot{h^0(\mathbb{P}(\varpi),\mathscr{O}_{\mathbb{P}(\varpi)}(k))\choose
2}.
$$
From this we obtain the following  inequality
$$
k[\mathscr{N}(k)-h^0(\mathbb{P}(\varpi),\mathscr{O}_{\mathbb{P}(\varpi)}(k))]\leq(d-1)
\cdot{h^0(\mathbb{P}(\varpi),\mathscr{O}_{\mathbb{P}(\varpi)}(k))\choose
2}.
$$
Now , if we suppose that
$\mathscr{N}(k)-h^0(\mathbb{P}(\varpi),\mathscr{O}_{\mathbb{P}(\varpi)}(k))\geq{h^0(\mathbb{P}(\varpi),\mathscr{O}_{\mathbb{P}(\varpi)}(k))\choose
2}$, we conclude that $k\leq d-1.$
\subsection{Proof of theorem \ref{3}}

Part $i)$: Let $\F$ be a holomorphic foliation on the complex
surface $X$ and $S$ a non-dicritical separatrix . In this conditions
M. Brunella in \cite{B1}\cite{B2} that  $$N_{\F}\cdot S-S\cdot
S\geq0. \ \ \ \ (1)$$

 We have that
$K_{\mathbb{P}(\varpi_0,\varpi_1,\varpi_2)}=T^*_{\F}\otimes
N^*_{\F}$. Since $T^*_{\F}=\mathcal{O}_{\mathbb{P}(\varpi)}(d-1)$
and
$K_{\mathbb{P}(\varpi_0,\varpi_1,\varpi_2)}=\mathcal{O}_{\mathbb{P}(\varpi)}(-\varpi_0-\varpi_1-\varpi_2)$
then
$N_{\F}=\mathcal{O}_{\mathbb{P}(\varpi)}(d+\varpi_0+\varpi_1+\varpi_2-1)$.
Hence
\begin{center}
$ N_{\F}\cdot
S=\dfrac{deg(S)(d+\varpi_0+\varpi_1+\varpi_2-1)}{\varpi_0\varpi_1\varpi_2}$\
\  and $S\cdot S =\dfrac{deg(S)^2}{\varpi_0\varpi_1\varpi_2} $
\end{center}
 By inequality  $(1)$ we get
$$\dfrac{deg(S)^2}{\varpi_0\varpi_1\varpi_2}\leq
\dfrac{deg(S)(d+\varpi_0+\varpi_1+\varpi_2-1)}{\varpi_0\varpi_1\varpi_2},$$
therefore $deg(S)\leq deg(\F)+\varpi_0+\varpi_1+\varpi_2 -1.$
\\
\\
Part $ii):$ If $S$ is quasi-smooth follows from theorem \ref{number}
that
$$
\mathscr{N}(\F)=\sum_{p\in Sin(\F)\cap
S}\mu_p^{orb}(\F)=deg(S)\cdot\frac{deg(\F)+\varpi_0+\varpi_1+\varpi_2-(deg(S)+1)}{\varpi_0\varpi_1\varpi_2}.
$$
On the other hand, we have that $S\cap Sing(\F)\neq\emptyset$. In
fact, follows from  Camacho-Sad index theorem \cite{B1}\cite{B2},
that
$$\sum_{p\in
Sin(\F)\cap S}CS(\F,S,p)=S\cdot S
=\dfrac{deg(S)^2}{\varpi_0\varpi_1\varpi_2} >0.$$ Hence, we can to
conclude that $\mathscr{N}(\F)>0$ and this in turn implies that
$$deg(S)<deg(\F)+\varpi_0+\varpi_1+\varpi_2-1.$$

\subsection{Proof of theorem \ref{number}}

Since $\mathscr{V}$ is quasi-smooth we have the adjunction formula
$\mathcal{N}_{\mathscr{V}}^*\simeq
\mathscr{O}_{\mathbb{P}(\varpi)}(-\deg(\mathscr{V}))_{|_{\mathscr{V}}}$,
see \cite{BGN}. On the other hand, we have the exact sequence of
orbibundle
$$
0\rightarrow T\mathscr{V}\rightarrow
T\mathbb{P}(\varpi)_{|_{\mathscr{V}}}\rightarrow
\mathcal{N}_{\mathscr{V}}\rightarrow 0.
$$
Hence, it follows that $c(T\mathbb{P}(\varpi))=c(T\mathscr{V})\cdot
\mathcal{N}_{\mathscr{V}}=c(T\mathscr{V})\cdot c(
\mathscr{O}_{\mathbb{P}(\varpi)}(\deg(\mathscr{V}))_{|_{\mathscr{V}}})$.
Hence
$$c_i(T\mathscr{V})=c_i(\mathbb{P}(\varpi))-c_{i-1}(T\mathscr{V})c_1(\mathscr{O}_{\mathbb{P}(\varpi)}(\deg(\mathscr{V}))_{|_{\mathscr{V}}}),$$
and inductively we get
$$
c_i(T\mathscr{V})=\left[\sum_{k=0}^{i}(-1)^k\sigma_{i-k}
(\varpi_0,\dots,\varpi_n)deg(\mathscr{V})^{k}\right]c_1(\mathscr{O}_{\mathbb{P}(\varpi)}(1)_{|_{\mathscr{V}}})^i.
$$
Since $\mathscr{V}$ is invariant by $\F$ we have that
$\F_{|_{\mathscr{V}}}$ is induced by a section $\vartheta$ of the
orbibundle
$T\mathscr{V}\otimes\mathscr{O}_{\mathbb{P}(\varpi)}(d-1)_{|_{\mathscr{V}}}$.
As it was done in the proof of the theorem \ref{1} we use the
 intersection theory for to conclude that the Chern class
$$c_{n-1}(T\mathscr{V}\otimes\mathscr{O}_{\mathbb{P}(\varpi)}(d-1)_{|_{\mathscr{V}}})=
\sum_{i=0}^{n-1}c_i(T\mathscr{V})c_1(\mathscr{O}_{\mathbb{P}(\varpi)}(d-1)_{|_{\mathscr{V}}})^{n-1-i}.$$
is the intersection product between the graph of
$\vartheta_{|_{\mathscr{V}}}$ and the graph of the null section, and
each singularity $p\in Sing(\F)$ gives a contribution equal to
$\mu_p^{orb}(\F)$. Thus
$$
\begin{array}{ccl}
  \displaystyle\sum_{p\in Sin(\F)\cap \mathscr{V}
}\mu_p^{orb}(\F)&=&\displaystyle\int\limits_{\mathscr{V}}^{orb}c_{n-1}(T\mathscr{V}\otimes\mathscr{O}_{\mathbb{P}(\varpi)}(d-1)_{|_{\mathscr{V}}})
\\
  &=&\displaystyle\sum_{i=0}^{n-1}\displaystyle\int\limits_{\mathscr{V}}^{orb}c_i(T\mathscr{V})c_1(\mathscr{O}_{\mathbb{P}(\varpi)}(d-1)_{|_{\mathscr{V}}})^{n-1-i} \\
\\
   & = &\dfrac{\displaystyle\sum_{i=0}^{n-1}\left[\sum_{k=0}^{i}(-1)^k\sigma_{i-k}
(\varpi_0,\dots,\varpi_n)deg(\mathscr{V})^{k+1}\right](d-1)^{n-1-i}}{\varpi_0\cdots\varpi_n}.
\end{array}
$$
\\
\\
\textbf{Acknowledgement}:
\\
\\
I would like to be thankful to Marcio G. Soares for  supervising my
work and for interesting conversations.


{\footnotesize
\begin{thebibliography}{99}

\addcontentsline{toc}{section}{Refer\^ecias}



\bibitem[Ab]{Ab}Abd' Allah Al Amrami, \emph{Cohomological study of weighted projective
space}; In Algebraic Geometry, Lecture Notes in pure and applied
mathematics, 193, Edited by Sinan Sertöz.



\bibitem[Au]{Au}
L. Autonne, Sur la théorie des équations différentielles du premier
ordre et du premier degré, J. École Polytech. \textbf{61} (1891)
35-122,\textbf{62} (1892) 47-180

\bibitem[BBFK]{BBFK}G. Barthel, J-P. Brasselet, K-H. Fieseler et L. Kaup; \emph{Diviseurs
invariants et homomorphisme de Poincaré des varietes toriques
complexes}. Tôhoku Math. J. \textbf{48} (1996), 363-390.



\bibitem[BGN]{BGN} C. Boye, K. Galicki, M. Nakamaye, \emph{Sasakian geometry, homotopy spheres and positive Ricci
curvature}; Topology \textbf{42} (2003), 981-1002



\bibitem[B-1]{B1} M. Brunella. \emph{Some remarks on indices of holomorphic vector
fields}; Publicacions Mathemàtiques \textbf{41}: 527-544, 1997.

\bibitem[B-2]{B2} M. Brunella. \emph{ Foliations on complex projective surfaces}; arXiv:math/0212082v1 [math.CV] 5 Dec
2002.



\bibitem[Ca]{Ca}M.M. Carnicer. \emph{The Poincaré problem in the non-dicritical case};
Ann. Math. \textbf{140}: 289-294,  1994.


\bibitem[CN]{CN}
D. Cerveau and A. Lins Neto, \emph{Holomorphic foliations in
$\mathbb{P}_{\mathbb{C}}^2$ having an invariant algebraic curve},
Ann. Inst. Fourier \textbf{41} (1991), 883-903.



\bibitem[CM]{CM} D. Cerveau, J.F. Mattei, \emph{Formes intégrables holomorphes
singularières}; Asterisque, 97 (1983)

\bibitem[C]{C}D. Cox, \emph{The homogeneous coordinate of a toric}; J. Algebraic
Geom. \textbf{4} (1995), 17-50.






\bibitem[Da]{Da}G. Darboux, \emph{Mémoire sur les équations différentielles algébriques du
premier ordre et du premier degré} (Mélanges), Bull. Sci. Math.
\textbf{2} (1878) 60-96, 123-144,151-200.

\bibitem[F]{F}
W. Fulton, Intersection Theory, second edition, Springer, 1998.

\bibitem[GH]{GH}
P. Griffiths, J. Harris. \emph{Principles of algebraic geometry};
John Wiley $\&$ Sons, New York, 1978.

\bibitem[H]{H}R. Hartshorne,\emph{ Algebraic Geometry}; Graduate Texts in Math. 52,
Springer Verlag, New York etc., 1977.

\bibitem[J]{J} J. P. Joaunolou, \emph{ \'Equations de Pfaff alg\'ebriques}.
 Lecture Notes in Math. \textbf{708}, Springer, 1979.

\bibitem[N]{N}
A. Lins Neto, \emph{Some examples for Poincaré and Painleve problem}
Ann. Scient. Ec. Norm. Sup., 4e série, \textbf{35}, 2002, p. 231 a
266.

\bibitem[M]{M}E. Mann, \emph{Cohomologie quantique orbifold des espaces projectifs
à poids}; arXiv:math/0510331v1 [math.AG] 16 Oct 2005.


\bibitem[Pa]{Pa}P. Painlevé, \emph{Sur les intégrales algébrique des équations
differentielles du premier ordre} and \emph{Mémoire sur les
équations différentielles du premier ordre}, Oeuvres de Paul
Painlevé; Tome II, Éditions du Centre National de la Recherche
Scientifique, \textbf{15}, quai Anatole-France, 75700, Paris, 1974.


\bibitem[Pe]{Pe} J. V. Pereira, \emph{Vector Fields, Invariant Varieties and Linear Systems}.
Annales de L\'Institut Fourier 51, no.5 (2001), 1385-1405.

\bibitem[C-L-P]{C-L-P}
J. V, Pereira,  C. Christopher, J.  Llibre, \emph{Multiplicity of
Invariant Algebraic Curves in Polynomial Vector Fields}; Pacific
Journal of Mathematics, v. 229, p. 63-117, 2007.

\bibitem[P]{P} H. Poincaré, \emph{Sur l'intégration algébrique des équations
différentielles du premier ordre et du premier degré}; Rend. Circ.
Mat. Palermo 5 (1891), 161-191.





\bibitem[MS]{MS}
M. G. Soares,\emph{ The Poincaré problem for hypersurfaces invariant
by one-dimensional foliations}, Inventiones Mathematicae, Alemanha,
v. 128, p. 495-500, 1997.












\bibitem[S]{S}I. Satake. \emph{On a generalization of the notion of manifold}; Proc. of the Nat.
Acad. of Sc. U.S.A. 42 (1956), pp. 359-363.


































\end{thebibliography}}
\end{document}